\documentclass[12pt,leqno,a4paper]{amsart}
\usepackage{amssymb, amsmath, enumerate,}
\usepackage{enumerate,color}

\newtheorem{theorem}{Theorem}[section]
\newtheorem{lemma}[theorem]{Lemma}
\newtheorem{corollary}[theorem]{Corollary}

\newtheorem*{thmA}{Theorem A}
\newtheorem*{thmB}{Theorem B}

\theoremstyle{remark}
\newtheorem{remark}[theorem]{Remark}

\numberwithin{equation}{section}

\begin{document}
\title[Supersolvable subgroups of order divisible by $3$]{Supersolvable subgroups of order divisible by $3$}
    \author[Beltr\'an and Shao]{Antonio Beltr\'an\\
     Departamento de Matem\'aticas\\
      Universitat Jaume I \\
     12071 Castell\'on\\
      Spain\\
     \\Changguo Shao \\
College of Science\\ Nanjing University of Posts and Telecommunications\\
     Nanjing 210023 Yadong\\
      China\\
     }

 \thanks{Antonio Beltr\'an: abeltran@uji.es ORCID ID: https://orcid.org/0000-0001-6570-201X \newline
 \indent Changguo Shao: shaoguozi@163.com ORCID ID: https://orcid.org/0000-0002-3865-0573}

\keywords{supersolvable groups, maximal subgroups, simple groups}

\subjclass[2010]{20D05, 20E28}

\begin{abstract}
We determine the structure of the finite non-solvable groups of order divisible by $3$ all whose maximal subgroups of order divisible by $3$ are supersolvable. Precisely, we demonstrate that if $G$ is a finite non-solvable group  satisfying the above condition on maximal subgroups, then either $G$ is a $3'$-group or $G/{\bf O}_{3'}(G)$ is isomorphic to ${\rm PSL}_2(2^p)$ for an odd prime $p$, where ${\bf O}_{3'}(G)$ denotes the largest normal $3'$-subgroup of $G$. Furthermore, in the latter case,  ${\bf O}_{3'}(G)$ is nilpotent and ${\bf O}_2(G)\leq {\bf Z}(G)$.
\end{abstract}

\maketitle

\section{Introduction}

In this paper, all groups are assumed to be finite, and we follow standard notation (e.g.
 \cite{Robinson}). The structure of a finite group is influenced  to a  large extent by the properties of some or all of its maximal subgroups.
 Notable instances of this phenomenon  include the solvability of a group having an odd-order nilpotent maximal subgroup, and the  characterization of the minimal non-nilpotent groups and the minimal non-supersolvable groups. Related to supersolvability, in recent years there has been growing interest in studying groups  that possess specific supersolvable subgroups. For instance, using the Feit-Thompson Theorem and the fact that minimal non-2-nilpotent groups are solvable, it easily follows that groups whose subgroups of even order are supersolvable are solvable too. Further properties of such groups are provided in \cite{Meng}.  Another example appears in \cite{Ballester2}, where  groups  in which every  maximal subgroup is supersolvable or normal are studied. Likewise, groups with less than six non-supersolvable subgroups are  proved to be solvable  in \cite{Ballester}.

   \medskip
   The aim of this paper is to further extend the class of minimal non-supersolvable  groups. We seek to investigate whether it is possible to determine the structure of groups whose maximal subgroups of order divisible by a fixed odd prime $p$ are all supersolvable. Of course, such groups need  neither  be solvable nor $p$-solvable, even when $p$ divides their orders. In fact, ${\rm PSL}_2(8)$, ${\rm Sz}(8)$ and ${\rm PSL}_2(7)$ are examples of simple groups satisfying the aforementioned conditions for $p=3,5$ and $7$, respectively.

  \medskip
  The initial step in addressing the problem posed is to be able to identify the non-abelian simple groups  that satisfy the stated hypotheses. We determine these groups when $p=3$.

 \begin{thmA}
 Let $G$ be a finite  non-abelian simple group such that every maximal subgroup of order divisible by $3$ is supersolvable. Then $G$ is isomorphic to ${\rm Sz}(2^{2n+1})$ with $n \geq 1$, or to ${\rm PSL}_2(2^p)$ with $p$ an odd prime.
\end{thmA}

    The fact that the Suzuki groups  of Lie type, ${\rm Sz}(q)$, are the only non-abelian simple groups whose order  is not divisible by $3$ allows us to reduce the arguments  required to establish the structure of the non-solvable groups  under study when $p=3$.
  It  may appear unexpected that ${\rm Sz}(q)$  is not implicated in  the structure of the groups of our main result, Theorem B. We denote by ${\bf F}(G)$ the Fitting subgroup of $G$, by $\Phi(G)$, the Frattini subgroup, and if $\pi$ is a set of prime numbers, then ${\bf O}_\pi(G)$  denotes the $\pi$-radical of $G$, that is, the largest normal $\pi$-subgroup of $G$.

\begin{thmB}

 Let $G$ be a finite non-solvable group of order divisible by $3$ such that every maximal subgroup of order divisible by $3$ is supersolvable.  Then ${\bf F}(G)=\Phi(G) ={\bf O}_{3'}(G)$ and
                                    $G/{\bf O}_{3'}(G)  \cong  {\rm PSL}_2(2^p)$, with $p$ an odd prime. Furthermore, ${\bf O}_2(G)\leq {\bf Z}(G)$.
\end{thmB}

In order to prove our results, we make use of several results based on the Classification of Finite Simple Groups. More precisely,
we appeal to a variant of a result  \cite[Theorem 1]{Barry} that, at first sight, seems obvious, but it is not: Every non-abelian simple group contains a subgroup which is itself a minimal simple group.
Similarly, we require information on the subgroup structure and the maximal subgroups of certain simple groups, for which we refer to distinct sources, namely \cite{BHR, Con,  HB3,  Lieb2, Wil}.

\medskip
We should note that the problem has not been addressed for any other odd prime number $p\neq 3$ because the set of simple groups that have $p'$-order can be much larger, and even currently indeterminate for groups of Lie type.

\section{Preliminaries}

Recall that a minimal simple group is a non-abelian simple group all of whose proper subgroups  are solvable. The classification of minimal simple groups  is a classic result due to Thompson, which is needed for our purposes.

\begin{lemma}[\cite{Thom}] \label{Thom} Let $G$ be a minimal simple group. Then $G$ is isomorphic to one of the following:
\begin{itemize}
\item[(1)] ${\rm PSL}_3(3)$;
\item[(2)] the Suzuki simple group ${\rm Sz}(2^p)$, where $p$ is an odd prime;
\item[(3)] ${\rm PSL}_2(p)$, where $p$ is a prime with $p>3$ and $5   \nmid $ $ p^2-1$;
\item[(4)] ${\rm PSL}_2(2^p)$, where $p$ is a prime;
\item[(5)] ${\rm PSL}_2(3^p)$, where $p$ is an odd prime.
\end{itemize}

\end{lemma}

In the next lemma,  we  detail the structure of the   normalizers  of the Sylow $2$-subgroups in the Suzuki simple groups of Lie type, ${\rm Sz}(q)$.

\begin{lemma}\label{lemma2}  Let $G={\rm Sz}(q)$, where $q=2^{2n+1}$ and $n\geq 1$.  If $P$ is a Sylow $2$-subgroup of $G$, then ${\bf N}_G(P)=P\rtimes C_{q-1}$ is a Frobenius group with kernel $P$ and complement $C_{q-1}$. In particular, ${\bf N}_G(P)$ is not supersolvable.
\end{lemma}

\begin{proof} This follows from \cite[Chap. XI. Lemma 3.1]{HB3}.
\end{proof}

\section{Proofs}

As said in the Introduction, our first objective is to determine all non-abelian simple groups that satisfy our conditions for $p=3$. The strategy consists in appealing to the minimal simple groups and prove first the following  variant of  Theorem 1 of \cite{Barry}. We remark that our proof differs from that of \cite{Barry}.

\begin{theorem}\label{minimal}
 If $G$ is a finite non-abelian simple group that is non-isomorphic to $Sz(q)$, then $G$ contains a subgroup which is a minimal simple group distinct from $Sz(2^p)$ with $p$ an odd prime.
\end{theorem}

\begin{proof} According to the Classification of Finite Simple Groups we distinguish three cases.

\medskip
(1) $G$ is a sporadic simple group.  All sporadic simple groups except $O'N,J_1,Ru$ and $J_3$ contain either $M_{12}$ or $M_{22}$ \cite{Con}, and  both $M_{12}$ and $M_{22}$ contain ${\rm PSL}_2(5)\cong A_5$. On the other hand, again by \cite{Con}, we know that $J_1$ and $O'N$ both contain  ${\rm PSL}_2(11)$; likewise,  $J_3$ contains ${\rm PSL}_2(19)$, and $Ru$ contains ${\rm PSL}_2(5)$.  All these subgroups indicated  contain ${\rm PSL}_2(5)$ by a theorem of Dickson \cite[II.8.27]{Hup}, so we are done.

\medskip
(2) $G$ is an Alternating group $A_n$ with $n\geq5$. It is clear that $G\geq A_5$.

\medskip

(3) $G$ is a simple group of Lie type over the field of $q$ elements, where $q=r^s$ with $r$ prime and $s\geq 1$.  Suppose first that $G$ is a classical simple group.  Let $l$ be Lie rank of $G$ and write $G:=G_l(q)$.  Note that $G\geq G_l(r^t)$, whenever $t$ is a prime divisor of $s$ or $t=1$, and also $G\geq G_{l-1}(q)$. Thus, we may assume that both $n\leq l$ and $t$ are minimal with respect to $M:=G_n(r^t)$ being a simple group.  Of course, $M\leq G$. Hereafter, we  will show on a case-by-case basis that $M$ always contains a minimal simple group that is distinct from ${\rm Sz}(q)$, or equivalently, whose order is divisible by $3$.

\medskip
\begin{itemize}
\item[(3.1)] $G\cong {\rm PSL}_n(q)$. Then $M\cong {\rm PSL}_2(r^t)$.  It is clear that $r$ must be a prime greater than $3$ when $t=1$. If, in addition, $5$ does not divide $r^2-1$ then, by Lemma \ref{Thom}, we know that ${\rm PSL}_2(r)$ is minimal simple, so we are done. If,  on the contrary, $5$ divides $r^2-1$, then again by \cite[II.8.27]{Hup}, ${\rm PSL}_2(r)$ contains ${\rm PSL}_2(5)\cong A_5$, so we are finished too. On the other hand, we must have $r=2$ whenever $t$ is an odd prime, and this case is finished again in view of Lemma \ref{Thom}.

\item[(3.2)] $G\cong {\rm P\Omega}_{2n+1}(q)$, where $n\geq3$ and $q$ is odd.  Then $M\cong {\rm P\Omega}_{7}(r)$. By \cite[Table 8.40]{BHR}, we have that $M$ contains a minimal simple subgroup ${\rm PSL}_2(5)$.

\item[(3.3)] $G\cong {\rm PSp}_{2n}(q)$, where $n\geq2$.  Then $M\cong {\rm PSp}_{4}(r^t)$. If $r$ is odd, it follows that $G$ contains a minimal simple subgroup ${\rm PSL}_2(r)$ by \cite[Table 8.12]{BHR}. But then, arguing as in (3.1), we can finish this case. If $r=2$, then $G$ contains a minimal simple subgroup ${\rm PSL}_2(2^t)$, where $t$ is an odd prime, by \cite[Table 8.14]{BHR}.

\item[(3.4)] $G\cong {\rm P\Omega}_{2n}^+(q)$, where $n\geq4$.  Then $M\cong  {\rm P\Omega}_{8}^+(r)$.  In this case, by \cite[Table 8.50]{BHR},  $G$ has a minimal simple subgroup ${\rm PSL}_2(5)$.

\item[(3.5)] $G\cong {\rm PSU}_{n+1}(q)$, where $n\geq2$.  Then $M\cong  {\rm PSU}_{3}(r^t)$. If $r$ is odd, then $t=1$.  By \cite[Table 8.5]{BHR}, we have that $G$ has a minimal simple subgroup ${\rm PSL}_2(r)$. The same argument used in (3.1) serves to get the conclusion in this case.

\item[(3.6)] $G\cong {\rm P\Omega}_{2n}^-(q)$, where $n\geq4$.  Then $M\cong   {\rm P\Omega}_{8}^-(r^t)$.  By \cite[Table 8.5]{BHR}, we know that ${\rm P\Omega}_{8}^-(r^t)$ has a subgroup ${\rm P\Omega}_{4}^-(r)\cong {\rm PSL}_2(r^2)$. Since ${\rm PSL}_2(r)\leq {\rm PSL}_2(r^2)$, we deduce that $G$ has a minimal simple subgroup ${\rm PSL}_2(r)$ if $r$ is odd, and then the same argument used in (3.1) finishes this case.  If $r=2$, then $t=1$ and $M\cong {\rm P\Omega}_{8}^-(2)$. By \cite[Table 8.52]{BHR}, we know that ${\rm PSp}_6(2)< M$. Moreover, according to \cite{Con}, we have ${\rm PSL}_2(7)\leq {\rm PSp}_6(2)$. Hence it follows that $G$ possesses a minimal simple subgroup ${\rm PSL}_2(7)$.

\end{itemize}

Next we consider the case  when $G$ is  an exceptional group. Similarly, we show that $M$ contains a minimal simple group of order divisible by $3$, that is,  non-isomorphic to ${\rm Sz}(q)$.
\begin{itemize}

\item[(3.7)] $G\cong G_2(q)$.   Then $M\cong  G_2(r^t)$.  If $r\neq 3$, then $M$ has a subgroup $G_2(2)$ by \cite[Table 4.1]{Wil}, and moreover  $G_2(2)$ has a minimal simple subgroup ${\rm PSL}_2(7)$. If $r=3$, then $M\cong G_2(3)$, and $G_2(3)$ contains a minimal simple subgroup ${\rm PSL}_2(13)$.

\item[(3.8)] $G\cong $ $^2G_2(q)$.   Then $r=3$ and $M\cong  $ $^2G_2(3)\cong {\rm PSL}_2(8)$ which is a minimal simple group.

\item[(3.9)] $G\cong $ $^3D_4(q)$.  Since $^3D_4(q)> G_2(q)$ by \cite[Theorem 4.3]{Wil}, as in (3.7), we also get the conclusion.

\item[(3.10)] $G\cong  F_4(q)$.  Then $M\cong  F_4(r)$.  By \cite[Theorem 4.4]{Wil}, we obtain that $M$ contains a minimal simple subgroup ${\rm PSL}_3(3)$.

\item[(3.11)] $G\cong$  $^2F_4(q)$.  Then $r=2$ and  $M\cong$  $^2F_4(2^3)$. But $^2F_4(2^3)>$ $^2F_4(2)$, and $^2F_4(2)$ contains a minimal simple subgroup ${\rm PSL}_3(3)$.

\item[(3.12)] $G\cong  E_6(q)$ or $G\cong$ $^2E_6(q)$.  Let $\alpha$ be a   graph automorphism of $G$, then ${\bf C}_G(\alpha)=F_4(q)$ by \cite[Section 7]{Gor}. Moreover, we know that $E_6(q)< E_7(q)< E_8(q)$. Thus, in all these cases, taking into account (3.10), we get the result.

\end{itemize}
\end{proof}

\begin{theorem} \label{minimal-simple}
 Let $G$ be a minimal simple group satisfying that every maximal subgroup of order divisible by $3$ is supersolvable. Then $G\cong {\rm PSL}_2(2^p)$  or $G\cong {\rm Sz}(2^p)$ with $p$ an odd prime.
\end{theorem}

\begin{proof}
As $G$ is a minimal simple group,  we have to discuss each of the groups listed in Lemma \ref{Thom}. Certainly, the order of ${\rm Sz}(2^p)$ with $p$ an odd prime is not divisible by $3$, so  ${\rm Sz}(2^p)$ trivially satisfies the condition of the theorem.  We can rule out the case ${\rm PSL}_3(3)$ because it possesses a maximal subgroup  isomorphic to $S_4$ (see \cite{Con}), which is not supersolvable. If $G\cong {\rm PSL}_2(p)$, where $p$ is a prime, then $G$ has maximal subgroups $A_4$ or $S_4$ by \cite[Table 8.1]{BHR}, providing again a contradiction. If $G\cong {\rm PSL}_2(3^p)$ with $p$ an odd prime, then $G$ also contains a subgroup ${\rm PSL}_2(3)\cong A_4$, a contradiction as well.

\medskip

Therefore, according to  Lemma \ref{Thom}, it only remains to show that  ${\rm PSL}_2(2^p)$
 with $p$ an odd prime does  satisfy the hypotheses of the theorem.  Note  that ${\rm PSL}_2(4)\cong {\rm PSL}_2(5)$ has been discarded above, so $p$ can be assumed to be odd indeed. Now,  by \cite[Table 8.1]{BHR} for instance, we know that the maximal subgroups of ${\rm PSL}_2(2^p)$ are either isomorphic to dihedral groups of order $2(2^p-1)$ and $2(2^p+1)$, or  isomorphic to $C_2^p\rtimes C_{2^p-1}$. However,  the fact that $p$ is odd implies that  only the dihedral groups of order $2(2^p+1)$ have  an order divisible by $3$. As these groups are supersolvable, we are finished.
\end{proof}

We are now ready to achieve  our first objective, which is an equivalent form of Theorem A.

\begin{theorem} \label{simple}

 Let $G$ be a  non-abelian  simple group of order divisible  by $3$ such that every maximal subgroup of order divisible by $3$ is supersolvable. Then $G$ is isomorphic to ${\rm PSL}_2(2^p)$ with $p$ an odd prime.
\end{theorem}

\begin{proof} We take into account that the Suzuki simple groups of Lie type are the only non-abelian simple groups of $3'$-order. If $G$ is not minimal simple, we can apply  Theorem \ref{minimal} and  deduce that $G$ has a proper (minimal) simple  subgroup of order divisible by $3$. This is a contradiction because such subgroup should be supersolvable by hypothesis, and obviously it is not. As a consequence, $G$ must be a minimal simple group, so we can apply  Theorem \ref{minimal-simple} and the result is proved.
\end{proof}

Once we have demonstrated Theorem A, we are able to prove Theorem B, which we state again.

\medskip

\begin{theorem} \label{main}
Let $G$ be a finite non-solvable group of order divisible by $3$ such that every maximal subgroup of $G$ is either supersolvable or a $3'$-group.  Then $\Phi(G)={\bf F}(G)={\bf O}_{3'}(G)$ and
                                    $G/{\bf O}_{3'}(G)  \cong  {\rm PSL}_2(2^p)$ with $p$ an odd prime. Furthermore ${\bf O}_2(G)\leq {\bf Z}(G)$.

\end{theorem}

\begin{proof} Let us denote $\overline{G}=G/S(G)\neq 1$, where $S(G)$ is the solvable radical of $G$. If $3$ divides $|S(G)|$, then it is clear that  every maximal subgroup $\overline{H}$ of $\overline{G}$ satisfies that $|H|$ is divisible by $3$, so by hypothesis $H$ is supersolvable, and thus, $\overline {H}$ too. This implies that $\overline{G}$ is either supersolvable or minimal non-supersolvable. Both possibilities lead to the solvability of $G$ by a well-known theorem of Doerk \cite{Doerk}, so we get a contradiction. Henceforth, we will assume that $3$ does not divide $|S(G)|$ for the rest of the proof. In particular, $S(G)\leq {\bf O}_{3'}(G)$.

\medskip
Next we claim that $S(G)= {\bf O}_{3'}(G)$. Suppose on the contrary that  $S(G)< {\bf O}_{3'}(G)$ and take $\overline{L}$ a minimal normal subgroup of $\overline{G}$, with $L\leq {\bf O}_{3'}(G)$. Since $\overline{L}$ is not solvable and the Suzuki simple group is the only simple group whose order is not divisible by $3$, it is clear that we can write  $\overline{L}=\overline{S}\times\ldots \times \overline{S}$, where $\overline{S}\cong {\rm Sz}(q)$ for some $q=2^a$ and $a\geq 2$. Now, take $\overline{P}$ a Sylow $2$-subgroup of $\overline{S}$, so $\overline{P_0}=\overline{P}\times \ldots \times \overline{P}$ is a Sylow $2$-subgroup of $\overline{L}$. By the Frattini argument, we have $\overline{G}=\overline{L}{\bf N}_{\overline{G}}(\overline{P_0})$. Notice that
$|{\bf N}_{\overline{G}}(\overline{P_0})|$ is divisible by $3$ and this subgroup is necessarily  proper in $\overline{G}$. Therefore, there exists some maximal subgroup $K$ of $G$ such that
${\bf N}_{\overline{G}}(\overline{P_0})\leq \overline{K}$. Then, by hypothesis, $K$ must be supersolvable, and as a consequence, ${\bf N}_{\overline{G}}(\overline{P_0})$ is supersolvable too.
In particular, we deduce that  ${\bf N}_{\overline{L}}(\overline{P_0})$ is supersolvable, and hence  ${\bf N}_{\overline{S}}(\overline{P})$ as well. This contradicts Lemma \ref{lemma2}, so the claim is proved.

\medskip
Next we prove that $\overline{G}$ is simple. Of course, we have that $3$ divides $|\overline{G}|$. Take again $\overline{L}$ a minimal normal subgroup of $\overline{G}$.
In this case, by the equality obtained in the above paragraph,  it is evident that $3$ divides $|\overline{L}|$, so we can write $\overline{L}= \overline{S_1}\times\ldots \times \overline{S_n}$, a  direct product of isomorphic non-abelian simple groups of order divisible by $3$.
If $\overline{L}<\overline{G}$, then the hypotheses imply that $L$ is supersolvable, which obviously is a contradiction. Thus $\overline{L}=\overline{G}$. Furthermore,  if $n>1$,  one can easily construct, for instance,  a maximal subgroup
of $\overline{G}$ (of order divisible by $3$) containing $\overline{S_i}$ for every $i=1,\ldots, n-1$. The non-solvability of such groups together with the hypotheses certainly lead   to a contradiction. Accordingly, $n=1$, that is, $\overline{G}$ is non-abelian simple, as wanted.
Now, it suffices to notice that the hypotheses of the theorem imply that  $\overline{G}$ satisfies the conditions of Theorem \ref{simple}, and  consequently, $G/{\bf O}_{3'}(G)  \cong  {\rm PSL}_2(2^p)$ with $p$ prime.

\medskip
We prove now that $G$ is a Frattini cover of ${\rm PSL}_2(2^p)$.  Suppose that there is a maximal subgroup $M$ of $G$  that does not contain ${\bf O}_{3'}(G)$.  We certainly have $M{\bf O}_{3'}(G)=G$ and note that $3$ must divide $|M|$. Then, by hypothesis $M$ is supersolvable,  which, together with the solvability of ${\bf O}_{3'}(G)$, yields to the solvability of $G$, a contradiction. As a result, we conclude that ${\bf O}_{3'}(G)\leq \Phi(G)\leq {\bf F}(G)$. The simplicity of $\overline{G}$ certainly implies the equality of these subgroups.

\medskip
Finally, we prove that ${\bf O}_2(G) \leq  {\bf Z}(G)$. We note
 first that $\Phi(G)$ is the only maximal normal subgroup of $G$. Indeed, let $N$ be any maximal normal subgroup of $G$. Since $N\Phi(G)<G$, then $\Phi(G)\leq N$, and the simplicity of $\overline{G}$ forces the equality $\Phi(G)=N$.   We can prove now that ${\bf O}_2(G) \leq  {\bf Z}(G)$. Let $P\neq 1$  be a Sylow $3$-subgroup of $G$. Then  ${\bf O}_2(G)P$ is supersolvable by hypothesis, so in particular $P \unlhd {\bf O}_2(G)P$. It follows  that ${\bf O}_2(G)$ centralizes every Sylow $3$-subgroup of $G$. Now, as ${\bf O}^{3'}(G)$ is generated by all Sylow $3$-subgroups of $G$, we have ${\bf O}_2(G)\leq {\bf C}_G({\bf O}^{3'}(G))$. But notice that   ${\bf O}^{3'}(G)=G$, otherwise ${\bf O}^{3'}(G)$ would be contained in $\Phi(G)$, a contradiction. We conclude that ${\bf O}_2(G) \leq {\bf Z}(G)$, so the proof is finished.
\end{proof}

\begin{remark}
Groups satisfying  the thesis of Theorem B do exist.  By \cite[Chap. B. Theorem 11.8]{Doerk-Hawkes}, given  a prime $q$ and a group $H$ whose order is divisible by $q$, there exists a group $G$ with a  normal, elementary abelian $q$-subgroup $N\neq 1$, such that $N\leq \Phi(G)$ and $G/N\cong H$. In particular, if we take $H={\rm PSL}_2(2^p)$ with $p$ prime, as $H$ is simple, we would have $N=\Phi(G)$. Therefore,  it is possible to ensure the existence of a group $G$ satisfying $G/\Phi(G)\cong {\rm PSL}_2(2^p)$, where $\Phi(G)$ is a elementary abelian $q$-subgroup for a prime $q$ dividing $|H|$. Furthermore, we  notice that if,  in addition, such group $G$  satisfies the hypotheses of Theorem B, then $q$ must be odd (and of course, distinct from $3$). Indeed, if $q=2$, then by Theorem B, we have $N={\bf Z}(G)$, so $G$ would be a perfect central extension of ${\rm PSL}_2(2^p)$. However, as the Schur multiplier of ${\rm PSL}_2(2^p)$ is trivial \cite{Con}, it certainly follows that $N=1$, a contradiction.
\end{remark}

We would like to remark that the condition on {\it every maximal subgroup}  given in Theorems A and B can  be replaced by just {\it every proper subgroup}. In fact, both conditions are equivalent. Thus, we also obtain the following.

\begin{corollary}
Let $G$ be a non-solvable group whose proper subgroups  are either  supersolvable or  $3'$-subgroups. Then either  $G$  is a $3'$-group or ${\bf F}(G)=\Phi(G)={\bf O}_{3'}(G)$ and ${\bf O}_2(G)\leq {\bf Z}(G)$ and $G/{\bf O}_{3'}(G)\cong {\rm PSL}_2(2^p)$, with $p$ prime.
\end{corollary}

\begin{proof}
If $G$ is not a $3'$-group, it is enough to apply Theorem \ref{main}.
\end{proof}

\noindent
{\bf Acknowledgements}
This work is supported by the National Nature Science Fund of China (No. 12471017 and No. 12071181) and the  first named author is  also supported by Generalitat Valenciana, Proyecto CIAICO/2021/193. The  second  named author is also supported by the  Natural Science Research Start-up Foundation of Recruiting Talents of Nanjing University of Posts and Telecommunications (Grant Nos. NY222090, NY222091).

\medskip
\noindent
{\bf Data availability} Data sharing not applicable to this article as no data sets were generated or analyzed during
the current study.

\bigskip
\noindent
{\bf \large Declarations}

\medskip
\noindent
{\bf Conflict of interest} The authors have no conflicts of interest to declare.

\bibliographystyle{plain}

\end{document}